\documentclass[12pt]{amsart}
\usepackage{longtable}
\usepackage{amsfonts,amssymb,amsmath,textcomp}

 2

\usepackage[autostyle]{csquotes}
\theoremstyle{plain}
\newtheorem{thm}{Theorem}[section]
\newtheorem{thma}{Theorem}

\newtheorem{prop}[thm]{Proposition}

\theoremstyle{definition}

\begin{document}
\baselineskip 6.1mm

\title
{Divisibility of the class numbers of imaginary quadratic fields}
\author{K. Chakraborty, A. Hoque, Y. Kishi, P. P. Pandey}

\address[K. Chakraborty]{Harish-Chandra Research Institute, HBNI, Chhatnag Road, Jhunsi, Allahabad-211019, India.}
\email{kalyan@hri.res.in}
\address[A. Hoque]{Harish-Chandra Research Institute, HBNI, Chhatnag Road, Jhunsi, Allahabad-211019, India.}
\email{azizulhoque@hri.res.in}
\address[Y. Kishi]{Department of Mathematics, Aichi University of Education, 1 Hirosawa Igaya-cho, Kariya, Aichi 448-8542, Japan.}
\email{ykishi@auecc.aichi-edu.ac.jp}
\address[P. P. Pandey]{Department of Mathematics, IISER Berhampur, Berhampur-760010,
Odisha, India.}
\email{prem.p2506@gmail.com}


\date{\today}


\begin{abstract}
For a given odd integer $n>1$, we provide some families of imaginary quadratic number fields
of the form $\mathbb{Q}(\sqrt{x^2-t^n})$ whose ideal class group has a subgroup
isomorphic to $\mathbb{Z}/n\mathbb{Z}$.
\end{abstract}

\maketitle{}

\section{Introduction}
The divisibility properties of the class numbers of number fields are very important
for understanding the structure of the ideal class groups of number fields.
For a given integer $n>1$, the Cohen-Lenstra heuristic \cite{CL84} predicts that a
positive proportion of imaginary quadratic number fields have class number divisible by $n$.
Proving this heuristic seems out of reach with the current state of knowledge.
On the other hand, many families of (infinitely many) imaginary quadratic fields
with class number divisible by $n$ are known.
Most of such families are of the type $\mathbb{Q}(\sqrt{x^2-t^n})$ or of the type
$\mathbb{Q}(\sqrt{x^2-4t^n})$, where $x$ and $t$ are positive integers with some 
restrictions (for the former see
\cite{AC55, IT11p, IT11, KI09, RM97, RM99, NA22, NA55, SO00, MI12},
and for the later see \cite{Cohn, GR01, IS11, IT15, LO09, YA70}).
Our focus in this article will be on the family $K_{t,x}=\mathbb{Q}(\sqrt{x^2-t^n})$.

In 1922, T. Nagell~\cite{NA22} proved that for an odd integer $n$,
the class number of imaginary quadratic field $K_{t,x}$
is divisible by $n$ if $t$ is odd, $(t,x)=1$,
and $q\mid x$, $q^2\nmid x$ for all prime divisors $q$ of $n$.
Let $b$ denote the square factor of $x^2-t^n$, that is,
$x^2-t^n=b^2d$, where $d<0$ is the square-free part of $x^2-t^n$.
Under the condition $b=1$, N. C. Ankeny and S. Chowla \cite{AC55}
(resp.\ M. R. Murty \cite[Theorem~1]{RM97})
considered the family $K_{3,x}$ (resp.\ $K_{t,1}$).
M. R. Murty also treated the family $K_{t,1}$ with $b<t^{n/4}/2^{3/2}$ (\cite[Theorem~2]{RM97}).
Moreover, K. Soundararajan \cite{SO00} (resp.\ A. Ito \cite{IT11p})
treated the family $K_{t,x}$ under the condition that
$b<\sqrt{(t^n-x^2)/(t^{n/2}-1)}$ holds (resp.\ all of divisors of $b$ divide $d$).
On the other hand, T. Nagell~\cite{NA55} (resp.\ Y. Kishi~\cite{KI09}, A. Ito~\cite{IT11}
and M. Zhu and T. Wang~\cite{AC55}) studied the family $K_{t,1}$
(resp.\ $K_{3,2^k}$, $K_{p,2^k}$ and $K_{t,2^k}$)
unconditionally for $b$, where $p$ is an odd prime.
In the present paper, we consider the case when both
$t$ and $x$ are odd primes and $b$ is unconditional and prove the following:

\begin{thm}\label{T1}
Let $n\geq 3$ be an odd integer and $p,q$ be distinct odd primes with $q^2<p^n$.
Let $d$ be the square-free part of $q^2-p^n$. Assume that $q \not \equiv \pm 1 \pmod {|d|}$. Moreover, we assume
$p^{n/3}\not= (2q+1)/3, (q^2+2)/3$ whenever both $d \equiv 1 \pmod 4$ and $3\mid n$.
Then the class number of $K_{p,q}=\mathbb{Q}(\sqrt{d})$ is divisible by $n$.
\end{thm}

In Table 1 (respectively Table 2), we list $K_{p,q}$ for small values of $p,q$ for $n=3$ (respectively for $n=5$).
It is readily seen from these tables that the assumptions in Theorem \ref{T1} hold very often. We can easily prove, by reading modulo $4$, that the condition \enquote{$p^{n/3}\not= (2q+1)/3, (q^2+2)/3$} in Theorem \ref{T1} holds whenever $p \equiv 3 \pmod 4$.
Further, if we fix an odd prime $q$, then the condition \enquote{$q \not\equiv \pm 1 \pmod{|d|}$} in Theorem \ref{T1} holds almost always, and,
this can be proved using the celebrated Siegel's theorem on integral points on affine curves. More precisely, we prove the following theorem in this direction.

\begin{thm}\label{T2}
Let $n\geq 3$ be an odd integer not divisible by $3$. For each odd prime $q$ the class number of $K_{p,q}$ is divisible by $n$ for all but finitely many $p$'s. Furthermore, for each $q$ there are infinitely many fields $K_{p,q}$. 
\end{thm}

\section{Preliminaries}
In this section we mention some results which are needed for the proof of the Theorem \ref{T1}. First we state a basic result from algebraic number theory.

\begin{prop}\label{P1}
Let $d \equiv 5 \pmod 8$ be an integer and $\ell$ be a prime. For odd integers $a,b$ we have 
$$\left(\frac{a+b\sqrt{d}}{2}\right)^{\ell} \in \mathbb{Z}[\sqrt{d}] \mbox{ if and only if } \ell=3.$$
\end{prop}
\begin{proof}
This can be easily proved by taking modulo some power of two.
\end{proof}

We now recall a result of Y. Bugeaud and T. N. Shorey \cite{BS01} on Diophantine equations which is one of the main ingredient in the proof of Theorem \ref{T1}.
Before stating the result of Y. Bugeaud and T. N. Shorey, we need to introduce some definitions and notations.

Let $F_k$ denote the $k$th term in the Fibonacci sequence defined by $F_0=0, \  F_1= 1$
and $F_{k+2}=F_k+F_{k+1}$ for $k\geq 0$. Similarly $L_k$ denotes the $k$th term in the Lucas
sequence defined by $L_0=2, \ L_1=1$ and $L_{k+2}=L_k+L_{k+1}$ for $k\geq 0$.
For $\lambda\in \{1, \sqrt{2}, 2\}$, we define the subsets $\mathcal{F}, \ \mathcal{G_\lambda},
\ \mathcal{H_\lambda}\subset \mathbb{N}\times\mathbb{N}\times\mathbb{N}$ by
\begin{align*}
\mathcal{F}&:=\{(F_{k-2\varepsilon},L_{k+\varepsilon},F_k)\,|\,
k\geq 2,\varepsilon\in\{\pm 1\}\},\\
\mathcal{G_\lambda}&:=\{(1,4p^r-1,p)\,|\,\text{$p$ is an odd prime},r\geq 1\},\\
\mathcal{H_\lambda}&:=\left\{(D_1,D_2,p)\,\left|\,
\begin{aligned}
&\text{$D_1$, $D_2$ and $p$ are mutually coprime positive integers with $p$}\\
&\text{an odd prime and there exist positive integers $r$, $s$ such that}\\
&\text{$D_1s^2+D_2=\lambda^2p^r$ and $3D_1s^2-D_2=\pm\lambda^2$}
\end{aligned}\right.\right\},
\end{align*}
except when $\lambda =2$, in which case the condition ``odd''
on the prime $p$ should be removed in the definitions of $\mathcal{G_\lambda}$
and $\mathcal{H_\lambda}$. 

\begin{thma}\label{A1}
Given $\lambda\in \{1, \sqrt{2}, 2\}$, a prime $p$ and positive co-prime integers $D_1$ and $D_2$, the number of positive integer solutions $(x, y)$ of the Diophantine equation
 \begin{equation}\label{E1}
  D_1x^2+D_2=\lambda^2p^y
 \end{equation}
is at most one except for $$
(\lambda,D_1,D_2,p)\in\mathcal{E}:=\left\{\begin{aligned}
&(2,13,3,2),(\sqrt 2,7,11,3),(1,2,1,3),(2,7,1,2),\\
&(\sqrt 2,1,1,5),(\sqrt 2,1,1,13),(2,1,3,7)
\end{aligned}\right\}
$$
and $(D_1, D_2, p)\in
\mathcal{F}\cup \mathcal{G_\lambda}\cup \mathcal{H_\lambda}$.
\end{thma}

We recall the following result of J. H. E. Cohn \cite{Cohn1} about appearance of squares in the Lucas sequence.
\begin{thma}\label{A2}
The only perfect squares appearing in the Lucas sequence are $L_1=1$ and $L_3=4$.
\end{thma}

\section{Proofs}
We begin with the following crucial proposition. 

\begin{prop}\label{P2}
Let $n,q,p,d$ be as in Theorem \ref{T1} and let $m$ be the positive integer with
$q^2-p^n=m^2d$. Then the element $\alpha =q+m\sqrt{d}$ is not an $\ell^{th}$ power of an element in the ring of integers of $K_{p,q}$ for any prime divisor $\ell $ of $n$.
\end{prop}

\begin{proof}
Let $\ell$ be a prime divisor of $n$. Since $n$ is odd, so is $\ell$.

We first consider the case when $d \equiv 2 \mbox{ or }3 \pmod 4$.
If $\alpha$ is an $\ell^{th}$ power, then there are integers $a,b$ such that 
$$q+m \sqrt{d}=\alpha =(a+b\sqrt{d})^{\ell}.$$
Comparing the real parts, we have
$$q=a^{\ell}+\sum_{i=0}^{(\ell-1)/2} \binom{\ell}{2i} a^{\ell-2i}b^{2i}d^i.$$
This gives $a\mid q$ and hence $a=\pm q$ or $a=\pm 1$.\\
Case (1A): $a= \pm q$.\\
We have $q+m\sqrt{d}=(\pm q+b\sqrt{d})^{\ell}$. Taking norm on both sides we obtain
$$p^n=(q^2-b^2d)^{\ell}.$$
Writing $D_1=-d>0$, we obtain
\begin{equation*}
D_1b^2+q^2=p^{n/ \ell}.
\end{equation*}
Also, we have 
\begin{equation*}
D_1m^2+q^2=p^n.
\end{equation*}
As $\ell$ is a prime divisor of $n$ so $(x,y)=(|b|,n/ \ell)$ and $(x,y)=(m,n)$ are
distinct solutions of (\ref{E1}) in positive integers for $D_1=-d>0,D_2=q^2, \lambda=1$.

Now we verify that $(1, D_1,D_2,p) \not \in \mathcal{E}$ and $(D_1,D_2,p) \not \in \mathcal{F}\cup \mathcal{G_\lambda}\cup \mathcal{H_\lambda}$. This will give a contradiction.
Clearly $(1,D_1,D_2,p) \not \in \mathcal{E}$. Further, as $D_1>3$, we see that $(D_1,D_2,p) \not \in \mathcal{G}_1$. From Theorem \ref{A2}, we see that $(D_1,D_2,p) \not \in \mathcal{F}$. Finally, if $(D_1,D_2,p) \in \mathcal{H}_1$ then there are positive integers $r,s$ such that 
\begin{equation}\label{eq:1}
3D_1s^2-q^2=\pm 1
\end{equation}
and
\begin{equation}\label{eq:2}
D_1s^2+q^2=p^r.
\end{equation}
By (\ref{eq:1}), we have $q\not= 3$, and hence we have $3D_1s^2-q^2=- 1$.
From this together with (\ref{eq:2}), we obtain
$$4q^2=3p^r+1,$$
that is $$(2q-1)(2q+1)=3p^r.$$
This leads to $2q-1=1 \mbox{ or } 2q-1=3$, but this is not possible as $q$ is an odd prime. Thus $(D_1,D_2,p) \not \in \mathcal{H}_1$.\\
Case (1B): $a=\pm 1$.\\
In this case we have $q+m\sqrt{d}=(\pm 1+b \sqrt{d})^{\ell}$.

Comparing the real parts on both sides, we get $q \equiv \pm 1 \pmod{|d|}$ which contradicts to the assumption \enquote{$q \not\equiv \pm 1 \pmod{|d|}$}.

Next we consider the case when $d \equiv 1 \pmod 4$.
If $\alpha$ is an $\ell^{th}$ power of some integer in $K_{p,q}$,
then there are rational integers $a,b$ such that
$$q+m\sqrt{d}=\left( \frac{a+b\sqrt{d}}{2} \right)^{\ell},\ \ a\equiv b\pmod 2.$$
In case both $a \mbox{ and }b$ are even, then we can proceed as in the case
$d \equiv 2 \mbox{ or }3 \pmod 4$ and obtain a contradiction under the assumption
$q \not\equiv \pm 1 \pmod{|d|}$. Thus we can assume that both $a$ and $b$ are odd.
Again, taking norm on both sides we obtain
\begin{equation}\label{E4}
4p^{n/\ell}=a^2-b^2d.
\end{equation}

Since $a,b$ are odd and $p \neq 2$, reading modulo 8 in (\ref{E4}) we get

$d \equiv 5 \pmod 8$. As $\left( \frac{a+b\sqrt{d}}{2} \right)^{\ell}=q+m\sqrt{d} \in \mathbb{Z}[\sqrt{d}]$, by Proposition \ref{P1} we obtain $\ell=3$.
Thus we have
$$q+m\sqrt{d}=\left( \frac{a+b\sqrt{d}}{2} \right)^3.$$
Comparing the real parts, we have
\begin{align}\label{E44}
8q&=a(a^2+3b^2d).
\end{align}
Since $a$ is odd, therefore, we have $a\in\{\pm 1,\pm q\}$.\\
Case (2A): $a=q$.\\
By (\ref{E44}), we have $8=q^2+3b^2d$, and hence, $2\equiv q^2\pmod{3}$.
This is not possible.\\
Case (2B): $a=-q$.\\
By (\ref{E4}) and (\ref{E44}), we have
$$4p^{n/3}=q^2-b^2d\ \text{and}\ 8=-(q^2+3b^2d).$$
From these, we have $3p^{n/3}=q^2+2$, which violates our assumption.\\
Case (2C): $a=1$.\\
By (\ref{E44}) and $d<0$, we have $8q=1+3b^2d<0$.
This is not possible.\\
Case (2D): $a=-1$.\\
By (\ref{E4}) and (\ref{E44}), we have
$$4p^{n/3}=1-b^2d\ \text{and}\ 8q=-(1+3b^2d).$$
From these, we have $3p^{n/3}=2q+1$, which violates our assumption.
This completes the proof.
\end{proof}

We are now in a position to prove Theorem \ref{T1}.

\begin{proof}[\bf Proof of Theorem~$\ref{T1}$]
Let $m$ be the positive integer with $q^2-p^n=m^2d$ and put $\alpha =q+m\sqrt{d}$.
We note that $\alpha$ and $\bar{\alpha}$ are co-prime and $N(\alpha)=\alpha \bar{\alpha}=p^n$.
Thus we get $(\alpha)= \mathfrak{a}^n$ for some integral ideal $\mathfrak{a}$ of $K_{p,q}$.
We claim that the order of $[\mathfrak{a}]$ in the ideal class group of $K_{p, q} $ is $n$.
If this is not the case, then we obtain an odd prime divisor $\ell$ of $n$ and an integer
$\beta $ in $K_{p,q}$ such that $(\alpha)=(\beta)^{\ell}$.
As $q$ and $p$ are distinct odd primes, the condition \enquote{$q \not \equiv \pm 1 \pmod{|d|}$}
ensures that $d<-3$. Also $d$ is square-free, hence the only units in the ring of integers of
$K_{p,q}=\mathbb{Q}(\sqrt{d})$ are $\pm1$. Thus we have $\alpha =\pm \beta^{\ell}$.
Since $\ell$ is odd, therefore, we obtain $\alpha= \gamma^{\ell}$ for some integer
$\gamma$ in $K_{p,q}$ which contradicts to Proposition \ref{P2}.
\end{proof}

We now give a proof of Theorem \ref{T2}. This is obtained as a consequence of a well known theorem of Siegel (see \cite{ES, LS}).

\begin{proof}[\bf Proof of Theorem $\ref{T2}$]
Let $n>1$ be as in Theorem \ref{T2} and $q$ be an arbitrary odd prime. For each odd prime $p \neq q$, from Theorem \ref{T1}, the class number of $K_{p,q}$ is divisible by $n$ unless $q\equiv \pm 1 \pmod {|d|}$.
If $q\equiv \pm 1 \pmod {|d|}$, then $|d|\leq q+1$. 

For any positive integer $D$, the curve 
\begin{equation}\label{E5}
DX^2+q^2=Y^n
\end{equation}
is an irreducible algebraic curve (see \cite{WS}) of genus bigger than $0$. From Siegel's theorem (see \cite{LS}) it follows that there are only finitely many integral points $(X,Y)$ on the curve (\ref{E5}). Thus, for each $d<0$ there are at most finitely many primes $p$ such that 
$$q^2-p^n=m^2d.$$
Since $K_{p,q}=\mathbb{Q}(\sqrt{d})$, it follows that there are infinitely many fields $K_{p,q}$ for each odd prime $q$. Further if $p$ is large enough, then for $q^2-p^n=m^2d$, we have $|d|>q+1$. Hence, by Theorem \ref{T1}, the class number of $K_{p,q}$ is divsible by $n$ for $p$ sufficiently large.
\end{proof}

\section{Concluding remarks}
We remark that the strategy of the proof of Theorem \ref{T1} can be adopted, together with the following result of W. Ljunggren \cite{LJ43}, to prove Theorem \ref{T4}. 
\begin{thma}\label{TE}
For an odd integer $n$, the only solutions to the Diophantine equation
\begin{equation}
\frac{x^n-1}{x-1}=y^2  
\end{equation}
in positive integers $x,y, n $ with $x>1$ is $n=5, x=3, y=11$.
\end{thma}

\begin{thm}\label{T4}
For any positive odd integer $n$ and any odd prime $p$, the class number of the imaginary quadratic field $\mathbb{Q}(\sqrt{1-p^n})$ is divisible by $n$ except for the case $(p, n)=(3, 5)$.  
\end{thm}
Theorem \ref{T4} alternatively follows from the work of T. Nagell (Theorem 25 in \cite{NA55})  which was elucidated by J. H. E. Cohn (Corollary 1 in \cite{CO03}). M. R. Murty gave a proof of Theorem \ref{T4} under condition either \enquote{$1-p^n$ is square-free with $n>5$} or \enquote{$m<p^{n/4}/2^{3/2}$ whenever $m^2\mid 1-p^n$ for some integer $m$ with odd $n>5$} (Theorem 1 and Theorem 2 in \cite{RM97}, also see 
\cite{RM99}).

Now we give some demonstration for Theorem \ref{T1}. All the computations in this paper were done using PARI/GP (version 2.7.6). Table 1 gives the list of imaginary quadratic fields $K_{p,q}$ corresponding to $n=3$, $p \leq 19$ (and hence discriminant not exceeding $19^3$). Note that the list does not exhaust all the imaginary quadratic fields $K_{p,q}$ of discriminant not exceeding $19^3$. Table 2 is the list of $K_{p,q}$ for $n=5$ and $p \leq 7$.\\
\begin{center}
\begin{longtable}{|l|l|l|l|l|l|l|l|l|l|}
\caption{Numerical examples of Theorem 1 for $n=3$.} \label{tab:long1} \\

\hline \multicolumn{1}{|c|}{$p$} & \multicolumn{1}{c|}{$q$} & \multicolumn{1}{c|}{$q^2-p^n$}& \multicolumn{1}{c|}{$d$} & \multicolumn{1}{c|}{$h(d)$}& \multicolumn{1}{|c|}{$p$} & \multicolumn{1}{c|}{$q$} & \multicolumn{1}{c|}{$q^2-p^n$}& \multicolumn{1}{c|}{$d$} & \multicolumn{1}{c|}{$h(d)$}\\ \hline 
\endfirsthead

\multicolumn{10}{c}%
{{\bfseries \tablename\ \thetable{} -- continued from previous page}} \\
\hline \multicolumn{1}{|c|}{$p$} & \multicolumn{1}{c|}{$q$} & \multicolumn{1}{c|}{$q^2-p^n$}& \multicolumn{1}{c|}{$d$} & \multicolumn{1}{c|}{$h(d)$}& \multicolumn{1}{|c|}{$p$} & \multicolumn{1}{c|}{$q$} & \multicolumn{1}{c|}{$q^2-p^n$}& \multicolumn{1}{c|}{$d$} & \multicolumn{1}{c|}{$h(d)$}\\ \hline 
\endhead
\hline \multicolumn{10}{|r|}{{Continued on next page}} \\ \hline
\endfoot
\hline 
\endlastfoot
 3&5&-2&-2&1*& 5&3&-116&-29&6\\
 5&7&-76&-19&1**& 7&3&-334&-334&12\\
 7&5&-318&-318&12&7&11&-222&-222&12\\
 7&13&-174&-174&12& 7&17&-54&-6&2*\\
 11&3&-1322&-1322&42&11&5&-1306&-1306&18\\
 11&7&-1282&-1282&12&11&13&-1162&-1162&12\\
 11&17&-1042&-1042&12&11&19&-970&-970&12\\
 11&23&-802&-802&12&11&29&-490&-10&2*\\
 11&31&-370&-370&12&11&37&-38&-38&6*\\
 13&3&-2188&-547&3&13&5&-2172&-543&12\\
 13&7&-2148&-537&12&13&11&-2076&-519&18\\
 13&17&-1908&-53&6&13&19&-1836&-51&2**\\
 13&23&-1668&-417&12&13&29&-1356&-339&6\\
 13&31&-1236&-309&12&13&37&-828&-23&3\\
 13&41&-516&-129&12&13&43&-348&-87&6\\
 13&47&-12&-3&1*&17&3&-4904&-1226&42\\
 17&5&-4888&-1222&12&17&7&-4864&-19&1**\\
 17&11&-4792&-1198&12&17&13&-4744&-1186&24\\
 17&19&-4552&-1138&12&17&23&-4384&-274&12\\
 17&29&-4072&-1018&18&17&31&-3952&-247&6\\
 17&37&-3544&-886&18&17&41&-3232&-202&6\\
 17&43&-3064&-766&24&17&47&2704&-1&1*\\
 17&53&-2104&-526&12&17&59&-1432&-358&6\\
 17&61&-1192&-298&6&17&67&-424&-106&6\\
 19&3&-6850&-274&12&19&5&-6834&-6834&48\\
 19&7&-6810&-6810&48&19&11&-6738&-6738&48\\
 19&13&-6690&-6690&72&19&17&-6570&-730&12\\
 19&23&-6330&-6330&48&19&29&-6018&-6018&48\\
 19&31&-5898&-5898&48&19&37&-5490&-610&12\\
 19&41&-5178&-5178&48&19&43&-5010&-5010&48\\
 19&47&-4650&-186&12&19&53&-4050&-2&1*\\
 19&59&-3378&-3378&24&19&61&-3138&-3138&24\\
 19&67&-2370&-2370&24&19&71&-1818&-202&6\\
 19&73&-1530&-170&12&19&79&-618&-618&12\\
 \end{longtable}
\end{center}

\begin{center}
\begin{longtable}{|l|l|l|l|l|l|l|l|l|l|}
\caption{Numerical examples of Theorem 1 for $n=5$.} \label{tab:long2} \\

\hline \multicolumn{1}{|c|}{$p$} & \multicolumn{1}{c|}{$q$} & \multicolumn{1}{c|}{$q^2-p^n$}& \multicolumn{1}{c|}{$d$} & \multicolumn{1}{c|}{$h(d)$}& \multicolumn{1}{|c|}{$p$} & \multicolumn{1}{c|}{$q$} & \multicolumn{1}{c|}{$q^2-p^n$}& \multicolumn{1}{c|}{$d$} & \multicolumn{1}{c|}{$h(d)$}\\ \hline 
\endfirsthead

\multicolumn{10}{c}%
{{\bfseries \tablename\ \thetable{} -- continued from previous page}} \\
\hline \multicolumn{1}{|c|}{$p$} & \multicolumn{1}{c|}{$q$} & \multicolumn{1}{c|}{$q^2-p^n$}& \multicolumn{1}{c|}{$d$} & \multicolumn{1}{c|}{$h(d)$}& \multicolumn{1}{|c|}{$p$} & \multicolumn{1}{c|}{$q$} & \multicolumn{1}{c|}{$q^2-p^n$}& \multicolumn{1}{c|}{$d$} & \multicolumn{1}{c|}{$h(d)$}\\ \hline 
\endhead
\hline \multicolumn{10}{|r|}{{Continued on next page}} \\ \hline
\endfoot
\hline 
\endlastfoot
 3&5&-218&-218&10&3&7&-194&-194&20\\
 3&11&-122&-122&10&3&13&-74&-74&10\\
 5&3&-3116&-779&10&5&7&-3076&-769&20\\
 5&11&-3004&-751&15&5&13&-2956&-739&5\\
 5&17&-2836&-709&10&5&19&-2764&-691&5\\
 5&23&-2596&-649&20&5&29&-2284&-571&5\\
 5&31&-2164&-541&5&5&37&-1756&-439&15\\
 5&41&-1444&-1&1*&5&43&-1276&-319&10\\
 5&47&-916&-229&10&5&53&-316&-79&5\\
 7&3&-16798&-16798&60&7&5&-16782&-16782&100\\
 7&11&-16686&-206&20&7&13&-16638&-16638&80\\
 7&17&-16518&-16518&60&7&19&-16446&-16446&100\\
 7&23&-16278&-16278&80&7&29&-15966&-1774&20\\
 7&31&-15846&-15846&160&7&37&-15438&-15438&80\\
 7&41&-15126&-15126&120&7&43&-14958&-1662&20\\
 7&47&-14598&-1622&30&7&53&-13998&-13998&100\\
 7&59&-13326&-13326&100&7&61&-13086&-1454&60\\
 7&67&-12318&-12318&60&7&71&-11766&-11766&120\\
 7&73&-11478&-11478&60&7&79&-10566&-1174&30\\
 7&83&-9918&-1102&20&7&89&-8886&-8886&60\\
 7&97&-7398&-822&20&7&101&-6606&-734&40\\
 7&103&-6198&-6198&40&7&107&-5358&-5358&40\\
 7&109&-4926&-4926&40&7&113&-4038&-4038&60\\
 7&127&-678&-678&20&&&&& \\
\end{longtable}
\end{center}

In both the tables we use $*$ in the column for class number to indicate the failure of condition \enquote{$q\not\equiv \pm 1 \pmod{|d|}$} of Theorem \ref{T1}. Appearance of $**$ in the column for class number indicates that both the conditions \enquote{$q\not\equiv \pm 1 \pmod{|d|}$ and $p^{n/3}\ne (2q+1)/3, (q^2+3)/3$} fail to hold. 
For $n=3$, the number of imaginary quadratic number fields obtained from the family provided by T. Nagell (namely $K_{t,1}$ with $t$ any odd integer) with class number divisible by $3$ and discriminant not exceeding $19^3$ are at most $9$, whereas, in Table 1 there are 59 imaginary quadratic fields $K_{p,q}$ with class number divisible by $3$ and discriminant not exceeding $19^3$ (Table 1 does not exhaust all such $K_{p,q}$). Out of these 59 fields in Table 1, the conditions of Theorem~\ref{T1} hold for 58. This phenomenon holds for all values of $n$. 

\noindent\textbf{Acknowledgements.}
The third and fourth authors would like to appreciate the hospitality provided by Harish-Chandra Research Institute, Allahabad, where the main part of the work was done. The authors would like to thank the anonymous referee for valuable comments to improve the presentation of this paper.


\begin{thebibliography}{99}

\bibitem{AC55} N. C. Ankeny and S. Chowla, {\it On the divisibility of the class number of quadratic fields}, Pacific J. Math. {\bf 5} (1955), 321--324.

\bibitem{BS01} Y. Bugeaud and T. N. Shorey, {\it On the number of solutions of the generalized Ramanujan-Nagell equation}, J. Reine Angew. Math. {\bf 539} (2001), 55--74.

\bibitem{CL84} H. Cohen and H. W. Lenstra Jr., {\it Heuristics on class groups of number fields}, in Number theory (Noordwijkerhout 1983), Lecture Notes in Math. {\bf 1068}, Springer, Berlin (1984),  33--62.

\bibitem{Cohn1} J. H. E. Cohn, {\it Square Fibonacci numbers, etc.}, Fibonacci Quart. {\bf 2} (1964), 109--113.

\bibitem{Cohn} J. H. E. Cohn, {\it On the class number of certain imaginary quadratic fields}, Proc. Amer. Math. Soc. {\bf 130} (2002), 1275--1277.

\bibitem{CO03} J. H. E. Cohn, {\it On the Diophantine equation $x^n=Dy^2+1$}, Acta Arith. {\bf 106} (2003), 73--83.

\bibitem{GR01} B. H. Gross and D. E. Rohrlich, {\it Some results on the Mordell-Weil group of the Jacobian of the Fermat curve}, Invent. Math. {\bf 44} (1978), 201--224.

\bibitem{ES} J. -H. Evertse and J. H. Silverman, {\it Uniform bounds for the number of solutions to $Y^n=f(X)$}, Math. Proc. Camb. Phil. Soc. {\bf 100} (1986), 237--248.

\bibitem{IS11} K. Ishii, {\it On the divisibility of the class number of imaginary quadratic fields}, Proc. Japan Acad. {\bf 87}, Ser. A (2011), 142--143.

\bibitem{IT11p} A. Ito, {\it A note on the divisibility of class numbers of imaginary quadratic fields $\mathbb{Q}(\sqrt{a^2-k^n})$}, Proc. Japan Acad. {\bf 87}, Ser. A (2011), 151--155.

\bibitem{IT11} A. Ito, {\it Remarks on the divisibility of the class numbers of imaginary quadratic fields $\mathbb{Q}(\sqrt{2^{2k}-q^n})$}, Glasgow Math. J. {\bf 53} (2011), 379--389.

\bibitem{IT15} A. Ito, {\it Notes on the divisibility of the class numbers of imaginary quadratic fields $\mathbb{Q}(\sqrt{3^{2e}-4k^n})$}, Abh. Math. Semin. Univ. Hambg. {\bf 85} (2015), 1--21.

\bibitem{KI09} Y. Kishi, {\it Note on the divisibility of the class number of certain imaginary quadratic fields}, Glasgow Math. J. {\bf 51} (2009), 187--191;
corrigendum, ibid. {\bf 52} (2010), 207--208.


\bibitem{LJ43} W. Ljunggren, {\it Some theorems on indeterminate equations of the form $\frac{x^n-1}{x-1}=y^q$}, Norsk Mat. Tidsskr. {\bf 25} (1943), 17--20.

\bibitem{LO09} S. R. Louboutin, {\it On the divisibility of the class number of imaginary quadratic number fields}, Proc. Amer. Math. Soc. {\bf 137} (2009), 4025--4028.

\bibitem{RM97}M. R. Murty, {\it The ABC conjecture and exponents of class groups of quadratic fields}, Contemporary Math. {\bf 210} (1998), 85--95.

\bibitem{RM99}M. R. Murty, {\it Exponents of class groups of quadratic fields}, Topics in number theory (University Park, PA, 1997), Math. Appl. {\bf 467}, Kluwer Acad. Publ., Dordrecht (1999), 229--239.

\bibitem{NA22} T. Nagell, {\it \"{U}ber die Klassenzahl imagin\"{a}r quadratischer, Z\"{a}hlk\"{o}rper}, Abh. Math. Sem. Univ. Hambg. {\bf 1} (1922), 140--150.

\bibitem{NA55} T. Nagell, {\it Contributions to the theory of a category of Diophantine equations of the second degree with two unknowns}, Nova Acta Sci. Soc. Upsal. Ser (4). {\bf 16} (1955), 1--38.

\bibitem{LS} C. L. Siegel, {\it Uber einige Anwendungen Diophantischer Approximationen}, Abh. Preuss. Akad. Wiss. Phys. Math. Kl. {\bf 1} (1929), 1-70; Ges. Abh., Band {\bf 1}, 209--266.

\bibitem{WS} W. M. Schmidt, Equations over finite fields, An elementary approach,
Lecture Notes in Mathematics, 536, Springer-Verlag, Berlin-New York, 1976.

\bibitem{SO00} K. Soundararajan, {\it Divisibility of class numbers of imaginary quadratic fields}, J. London Math. Soc. {\bf 61} (2000), 681--690.

\bibitem{YA70} Y. Yamamoto, {\it On unramified Galois extensions of quadratic number fields}, Osaka J. Math. {\bf 7} (1970), 57--76.

\bibitem{MI12} M. Zhu and T. Wang, {\it The divisibility of the class number of the imaginary quadratic field $\mathbb{Q}(\sqrt{2^{2m}-k^n})$}, Glasgow Math. J. {\bf 54} (2012),
149--154.

\end{thebibliography}
\end{document}